\documentclass[a4paper,12pt]{article}

\usepackage[utf8]{inputenc}
\usepackage[T1]{fontenc}
\usepackage{amsthm,amsmath,amssymb,bbm}
\usepackage{hyperref}
\usepackage{cleveref}
 \newtheorem{Thm}{Theorem}[section]
 \newtheorem{Lem}[Thm]{Lemma}

\theoremstyle{remark}
 
\theoremstyle{definition}
\newtheorem{Def}[Thm]{Definition}
\newtheorem{Expl}[Thm]{Example}

\newcommand\ol{\overline}

\newcommand\inv{^{-1}}

\newcommand{\End}{\operatorname{End}}
\newcommand{\ot}{\otimes}

\newcommand{\ou}[1]{\underset{{#1}}{\otimes}}
\newcommand\K{\mathbbm k}
\newcommand\Der{\operatorname{Der}}
\newcommand\chr{\operatorname{char}}
\newcommand\UU{{\mathfrak U}}
\newcommand\uu{{\mathfrak u}}
\newcommand\MS{\mathcal S}
\newcommand\ad{\operatorname{ad}}
\begin{document}
\title{A note on the restricted universal enveloping algebra of a restricted Lie-Rinehart Algebra}
\author{Peter Schauenburg\\
  Institut de Math{\'e}matiques de Bourgogne \\--- UMR 5584 du CNRS\\
Universit{\'e} de Bourgogne\\
BP 47870, 21078 Dijon Cedex\\
France\\
\texttt{peter.schauenburg@u-bourgogne.fr}}
\maketitle
\begin{abstract}
  Lie-Rinehart algebras, also known as Lie algebroids, give rise to Hopf algebroids by a universal enveloping algebra construction, much as the universal enveloping algebra of an ordinary Lie algebra gives a Hopf algebra, of infinite dimension. In finite characteristic, the universal enveloping algebra of a restricted Lie algebra admits a quotient Hopf algebra which is finite-dimensional if the Lie algebra is. Rumynin has shown that suitably defined restricted Lie algebroids allow to define restricted universal enveloping algebras that are finitely generated projective if the Lie algebroid is. This note presents an alternative proof and possibly fills a gap that might, however, only be a gap in the author's understanding.
\end{abstract}
\maketitle
\section{Introduction}
Throughout the paper, we fix a commutative base ring $\K$. All algebras, associative or Lie, and all unadorned tensor products or sets of homomorphisms are to be understood over $\K$. For the most part, we will assume that $\chr(\K)=p$ is prime.

The concept of a Lie algebroid or Lie-Rinehart algebra was introduced by several authors to capture the algebraic structure of a Lie algebra of vector fields. This is not only a Lie algebra but also a module over the algebra of functions, but not a Lie algebra over that algebra. Rather than being bilinear with respect to the algebra of functions, the bracket fulfills a relation involving the action of the vector fields as derivations on the functions. The axiomatics of a Lie-Rinehart algebra thus involve a commutative algebra $R$ and a Lie algebra $L$ which is a left $R$-module and acts on $R$ by derivations. Predictably a restricted Lie-Rinehart algebra is a restricted Lie algebra that is also an $R$-module and acts on $R$ by derivations such that the $p$-th power operation in $L$ corresponds to the $p$-th power of a derivation. The additional compatibility condition that has to be imposed between the $R$-module structure and the $p$-th power operation appears in work of Hochschild. Rinehart has defined the universal enveloping algebra of a Lie-Rinehart algebra, and proved a Poincaré-Birkhoff-Witt theorem for these envelopes. Rumynin has defined the restricted universal enveloping algebra of a restricted Lie-Rinehart algebra $L$ in the obvious way, and proved the corresponding Poincaré-Birkhoff-Witt theorem in the case that $L$ is projective: In a localization at a prime ideal, the restricted universal enveloping algebra is a free module with a PBW basis truncated at $p$-th powers. In particular, if $L$ is finitely generated projective as left $R$-module, then so is its universal enveloping algebra.

The present note provides an alternative proof of the PBW theorem for restricted Lie-Rinehart algebras. We use Rinehart's PBW theorem and adapt the technique used in Jacobson's textbook on Lie algebras to give a ``better'' basis of the universal envelope of a restricted Lie algebra. The main advantage of our proof is perhaps that it fits comfortably in this rather short note, wheras Rumynin opts to leave a ``long and instructive exercise'' to his readers. We also spend some time to show that localizing a restricted Lie-Rinehart algebra at a prime ideal of $R$ (or with respect to a multiplicative set for that matter) does yield a restricted Lie-Rinehart algebra; this seems an essential ingredient when treating the case when $L$ is a projective rather than a free $R$-module.

\section{Restricted Lie-Rinehart Algebras}
\label{sec:restr-lie-rineh}

Let $R$ be a commutative $\K $-algebra. Recall \cite{MR0055323,MR0125867,MR0154906} that a Lie-Rinehart algebra is a $\K$-Lie algebra $L$ which is also an $R$-module, endowed with an $R$-linear $\K$-Lie algebra homomorphism $\epsilon\colon L\to\Der_{\K}(R)$ called the anchor, such that $[x,ry]=r[x,y]+\epsilon(x)(r)y$ for $x,y\in L$ and $r\in R$. We will often suppress $\epsilon$ in the sequel. Recall also that a representation of a Lie-Rinehart algebra $L$ is an $R$-module $M$ with a left $R$-linear Lie algebra map $\rho\colon L\to\End_{\K}(M)$ such that $\rho(x)(rm)=r\rho(x)(m)+\epsilon(x)(r)m$.

\begin{Lem}\label{thm:1}
  Let $M$ be a representation of the $R$-Lie-Rinehart algebra $L$. Assume that $\chr(\K)=p$ is prime. Then $(rx)^p=r^px^p+(rx)^{p-1}(r)x$ in $\End_{\K}(M)$ for $r\in R$ and $x\in L$.
\end{Lem}
The lemma is a rephrasing of \cite[Lemma 1]{MR0070961}, and gives rise to the following definition, which is at least implicit in \cite{MR0070961} and explicit in \cite{MR1738261}:
\begin{Def}
  Assume $ \chr(\K)=p$ is prime. A restricted Lie-Rinehart algebra is a Lie-Rinehart algebra $L$ which is a restricted Lie algebra such that
  \begin{equation}
    \label{eq:1}
    (rx)^{[p]}=r^px^{[p]}+(rx)^{p-1}(r)x
  \end{equation}
for $x\in L$ and $r\in R$.

  A restricted representation of $L$ is a representation of $L$ both in the sense of Lie-Rinehart algebras and of restricted Lie algebras.
\end{Def}

\begin{Expl}
  If $ \chr(\K)=p$ is prime and $R$ a $\K$-algebra, then $\Der_{\K}(R)$ is a restricted Lie-Rinehart algebra.
\end{Expl}
\begin{Expl}
  Let $H$ be a $\times_R$-bialgebra (as defined in \cite{Tak:GAAA}) over a commutative $\K$-algebra $R$, such that $r=\ol r\in H$ for $r\in R$. The $R$-module $P(H)=\{X\in H|\Delta(X)=X\ot 1+1\ot X\}$ of primitive elements of $H$ is a Lie-Rinehart algebra. If $ \chr(\K)=p$ is prime, then $P(H)$ is a restricted Lie-Rinehart algebra with respect to the $p$-th power operation.
\end{Expl}
\begin{Expl}\label{ex:1}
  Let $A$ be a $\K$-algebra, and $R\subset A$ a subalgebra such that $ar-ra\in R$ for all $a\in A$ and $r\in R$. Then $A$ is an $R$-Lie-Rinehart algebra with respect to the commutator. If $\chr(\K)=p$ is a prime, then $A$ is restricted with respect to the $p$-th power operation.
\end{Expl}

The following observation is elementary but will be useful later:
\begin{Lem}\label{thm:4}
  Let $L,L'$ be two restricted Lie-Rinehart algebras over $R$ in characteristic $p$. Let $f\colon L\to L'$ be a homomorphism of Lie-Rinehart algebras. Assume that $f(x^{[p]})=f(x)^{[p]}$ for all $x$ in a generating set of the left $R$-module $L$. Then $f$ is restricted.
\end{Lem}
\begin{proof}
  If $f(x^{[p]})=f(x)^{[p]}$, then $f((rx)^{[p]})=f(r^px^{[p]}+(rx)^{p-1}(r)x)=r^pf(x^{[p]})+(rf(x))^{p-1}(r)f(x)=r^pf(x)^{[p]}+(rf(x))^{p-1}(r)f(x)=(rf(x))^{[p]}=f(rx)^{[p]}$. That a Lie algebra homomorphism $f$ between restricted Lie algebras is restricted if $f(x^{[p]})=f(x)^{[p]}$ for all $x$ in a generating set of the $\K$-module $L$ should be well known.
\end{proof}

\section{The restricted universal enveloping algebra}
\label{sec:restr-envel-algebra}
Recall from \cite{MR0154906} that the universal enveloping algebra $\UU(R,L)$ of a Lie-Rinehart algebra $(R,L)$ is a $\K$-algebra endowed with an algebra map $\iota_R\colon R\to \UU(R,L)$ and a Lie algebra map $\iota_L\colon L\to\UU(R,L)$ such that $\iota_R(r)\iota_L(x)=\iota_L(rx)$ and $\iota_L(x)\iota_R(r)-\iota_R(r)\iota_L(x)=\iota_R(\epsilon(x)(r))$ for $x\in L$ and $r\in R$, and universal among triples $(A,i_R,i_L)$ with analogous properties.

\begin{Def}[\cite{MR1738261}]\label{def:RUEA}
  Let $(R,L)$ be a restricted Lie-Rinehart algebra in characteristic $p$. The restricted universal enveloping algebra is a universal triple $(\uu(R,L),\iota_R,\iota_L)$ with an associative algebra $\uu(R,L)$, an algebra map $\iota_R\colon R\to\uu(R,L)$ and a restricted Lie algebra map $\iota_L\colon L\to\uu(R,L)$ such that $\iota_R(r)\iota_L(x)=\iota_L(rx)$ and $\iota_L(x)\iota_R(r)-\iota_R(r)\iota_L(x)=\iota_R(\epsilon(x)(r))$ for $x\in L$ and $r\in R$.
\end{Def}

We have already noted above that taking primitive elements takes us from certain $\times_R$-bialgebras to (restricted) Lie algebroids. In the other direction, the (restricted) universal envelope of a (restricted) Lie-Rinehart algebra is a $\times_R$-bialgebra. The relation is treated in detail in \cite{MR2653938}, but was already known in \cite{MR1738261} (in the restricted case); finding the earliest reference is perhaps rendered harder by the fact that numerous versions of the axiomatics of ``hopf algebroids'', ``quantum groupoids'', ``bialgebroids'' and the like are around. In a special case, the idea of a comultiplication on the envelope of a Lie algebroid can be tracked back to \cite{MR812990}. In each case, while the bialgebra-like properties are fruitful and interesting, they are rather immediately proved from the universal property of the universal enveloping algebras.

The existence of both the universal enveloping algebra and the restricted universal enveloping algebra is to be expected by the general principles of universal algebra. More interestingly, Rinehart proved a version of the Poincaré-Birkhoff-Witt theorem for universal enveloping algebras of Lie-Rinehart algebras. Rumynin \cite{MR1738261} has provided the appropriate version for the restricted case:
\begin{Thm}
  Let $(R,L)$ be a restricted Lie-Rinehart algebra in characteristic $p$.
  \begin{enumerate}
  \item If $L$ is a free $R$-module with ordered basis $(x_i)_{i\in I}$, then $\uu(R,L)$ is a free left $R$-module with basis $\{\iota_L(x_{i_1})^{e_1}\cdot\dots\cdot\iota_L(x_{i_\ell})^{e_\ell}|\ell\geq 0,i_1<\dots<i_\ell,1\leq e_i<p\}$.
  \item If $L$ is a projective $R$-module, then $\uu(R,L)$ is a locally free left $R$-module.
  \item If $L$ is a finitely generated projective $R$-module, then $\uu(R,L)$ is finitely generated projective as left $R$-module.
  \end{enumerate}
\end{Thm}
\begin{proof}
  For the first assertion we use Rinehart's PBW theorem for $\UU(R,L)$ which asserts that the ordered monomials in $\iota_L(x_i)$ form a basis of $\UU(R,L)$, and we adapt the technique from \cite[V.7]{MR559927} to pass to the ``small'' enveloping algebra.

  We note that $\UU=\UU(R,L)$ is filtered with respect to the total degree in the $\iota_L(x_i)$. By \cite{MR0154906} the ordered monomials in the $\iota_L(x_i)$ form a basis of the free left $R$-module $\UU$. In particular $\UU^{(n)}/\UU^{(n-1)}$ is free with basis the classes of those elements in the PBW basis of total degree $n$.

  The elements $z_i:=\iota_L(x_i)^p-\iota_L(x_i^{[p]})$ are central in $\UU(R,L)$ since \begin{multline*}\iota_L(x_i^{[p]})\iota_L(x_j)-\iota_L(x_j)\iota_L(x_i^{[p]})=\iota_L(\ad(x_i)^p(x_j))=\\=\ad(\iota_L(x_i))^p(\iota_L(x_j))=\iota_L(x_i)^p\iota_L(x_j)-\iota_L(x_j)\iota_L(x_i)^p
  \end{multline*}
  and $\iota_L(x_i^{[p]})\iota_R(r)=\iota_R(\epsilon(x_i^{[p]})(r))=\iota_R(\epsilon(x_i)^p(r))=\iota_R(r)\iota_L(x_i)^p$. We claim that the elements $z_{i_1}^{d_1}\cdot\dots\cdot z_{i_r}^{d_r}\iota_L(x_{i_1})^{e_1}\cdot\dots\cdot\iota_L(x_{i_r})^{e_r}$ with $r\geq 0$, $0\leq d_i$, $0\leq e_i<p$, and $d_i+e_i>0$, form another $R$-basis of $\UU$. For this it suffices to check that the classes of those monomials with $p\sum d_i+\sum e_i=n$ form an $R$-basis of $\UU^{(n)}/\UU^{(n-1)}$. But this is true since modulo $\UU^{(n-1)}$ these coincide with the elements of the PBW basis of $\UU^{(n))}$.

  Now (re)define $\uu(R,L)$ by factoring the ideal of $\UU(R,L)$ generated by positive powers of the $z_i$. By the above, $\uu(R,L)$ is clearly free with the claimed basis, but we need to verify that the maps $\iota_L$ and $\iota_R$ defined by prolonging those defined for $\UU(R,L)$ satisfy the properties in \cref{def:RUEA}. Except for the property that $\iota_L$ should be a map of restricted Lie algebras, these are already inherited from $\UU(R,L)$. In particular $\iota_L$ is a map of Lie-Rinehart algebras to $\uu(R,L)$, which is a restricted Lie-Rinehart algebra in the sense of \ref{ex:1}. Thus we are done by \ref{thm:4}.

  Now assume that $L$ is a projective left $R$-module. We need to show that each localization $R_{\mathfrak p}\ou R\uu(R,L)$ at a prime ideal $\mathfrak p$ of $R$ is a free $R_{\mathfrak p}$-module. Since $\uu(R,L)$ is clearly finitely generated if $L$ is, this finishes the proof.

  In \cite{MR0154906} it was already shown how to endow $R_{\mathfrak p}\ou RL$ with the structure of a Lie-Rinehart algebra over $R_{\mathfrak p}$. We will show in the next section how to extend the $p$-structure on $L$ to make $R_{\mathfrak p}\ou RL$ a restricted Lie-Rinehart algebra. Obviously $R_{\mathfrak p}\ou R\UU$ is the universal envelope of the latter, and free as an $R_{\mathfrak p}$-module, since $R_{\mathfrak p}\ou RL$ is free.
\end{proof}

\section{Restricted Lie-Rinehart algebras under localization}
\label{sec:restr-lie-rineh-2}

Let $(R,L)$ be a restricted Lie-Rinehart algebra in characteristic $p$, and $\mathfrak p\subset R$ a prime ideal. We will show how to endow $(R_{\mathfrak p},L_{\mathfrak p})$ with the structure of a restricted Lie-Rinehart algebra. In \cite{MR0154906} the structure of Lie-Rinehart algebra is already described. For simplicity of notations we put $\MS=R\setminus\mathfrak p$. Then for $r\in R$, $s,t\in \MS$ and $x,y\in L$ we have
$\epsilon(s\inv x)(t\inv r)=s\inv t\inv x(r)-s\inv t^{-2}x(t)r$ and $[s\inv x,t\inv y]=s\inv t\inv[x,y]+s\inv x(t\inv)y-t\inv y(s\inv)x$.

Now we will define the $p$-th power operation on $L_{\mathfrak p}$ by
\begin{equation}
  (s\inv x)^{[p]}=s^{-p}x^{[p]}-s^{-p-1}x^{p-1}(s)x=s^{-p-1}(sx^{[p]}-x^{p-1}(s)x)
\end{equation}
This is inevitable, as one sees easily by developing $(ss\inv x)^{[p]}$. The main issue is to see that the power operation is well-defined.

We first calculate in $\UU_{\mathfrak p}=\UU(R_{\mathfrak p},L_{\mathfrak p})$:
\begin{equation*}
  (ss\inv x)^p=s^p(s\inv x)^p+x^{p-1}(s)s\inv x,
\end{equation*}
and thus
\begin{equation*}
  (s\inv x)^p=s^{-p}x^p-s^{-p-1}x^{p-1}(s)x=s^{-p-1}(sx^p-x^{p-1}(s)x).
\end{equation*}
Assume that $t\inv y=s\inv x$, i.e. there exists $u\in\MS$ such that $utx=usy$ in $L$.
Then $(s\inv x)^p=(t\inv y)^p=t^{-p-1}(ty^p-y^{p-1}(t)y)$, which means that there is $v\in\MS$ with
\begin{equation}\label{eq:2}
  vt^{p+1}(sx^p-x^{p-1}(s)x)=vs^{p+1}(ty^p-y^{p-1}(t)y).
\end{equation}
On the other hand $(utx)^p=(usy)^p$, that is
\begin{equation*}
  u^pt^px^p+(utx)^{p-1}(ut)x=u^ps^py^p+(usy)^{p-1}(us)y.
\end{equation*}
Multiply this with $vts$ and subtract \eqref{eq:2} multiplied by $u^p$ to get
\begin{equation}
  \label{eq:3}
  vts(utx)^{p-1}(ut)x+vu^pt^{p+1}x^{p-1}(s)x=vts(usy)^{p-1}(us)y+vu^ps^{p+1}y^{p-1}(t)y.
\end{equation}
Next, we note that also $(utx)^{[p]}=(usy)^{[p]}$, that is
\begin{equation*}
  u^pt^px^{[p]}+(utx)^{p-1}(ut)x=u^ps^py^{[p]}+(usy)^{p-1}(us)y.  
\end{equation*}
Multiply this with $vts$ and subtract \eqref{eq:3} to obtain
\begin{equation*}
  vu^pt^{p+1}sx^{[p]}-vu^pt^{p+1}x^{p-1}(s)x=vu^pts^{p+1}y^{[p]}-vu^ps^{p+1}y^{p-1}(t)y,
\end{equation*}
or
\begin{equation*}
  vu^pt^{p+1}(sx^{[p]}-x^{p-1}(s)x)=vu^ps^{p+1}(ty^{[p]}-y^{p-1}(t)y), 
\end{equation*}
which implies
\begin{equation*}
  s^{-p-1}(sx^{[p]}-x^{p-1}(s)x)=t^{-p-1}(ty^{[p]}-y^{p-1}(t)y)
\end{equation*}
in $L_{\mathfrak p}$.

Next, we verify the $p$-power identity $(a+b)^{[p]}=a^{[p]}+b^{[p]}+\sum_{i=1}^{p-1}\lambda_i(a,b)$ with the usual expressions $\lambda_i(a,b)$ in the Lie algebra generated by $a,b\in L_{\mathfrak p}$. We can assume $a=s\inv x$ and $b=s\inv y$ for $x,y\in L$ and $s\in\MS$. Again, we pass through a calculation in $\UU_{\mathfrak p}$: We have
\begin{align*}
  (s\inv(x+y))^p&=(s\inv x)^p+(s\inv y)^p+\sum\lambda_i(s\inv x,s\inv y)\\
                &=s^{-p-1}(sx^p-x^{p-1}(s)x+sy^p-y^{p-1}(s)y)+\sum\lambda_i(s\inv x,s\inv y)\\
  \intertext{and on the other hand}
  (s\inv(x+y))^p&=s^{-p-1}(s(x+y)^p-(x+y)^{p-1}(s)(x+y))\\
                &=s^{-p-1}(sx^p+sy^p+s\sum\lambda_i(x,y)-(x+y)^{p-1}(s)(x+y))\\
  \intertext{so that combining the two}
  (s\inv(x+y))^{[p]}&=s^{-p-1}(s(x+y)^{[p]}-(x+y)^{p-1}(s)(x+y))\\
                &=s^{-p-1}(sx^{[p]}+sy^{[p]}+s\sum\lambda_i(x,y)-(x+y)^{p-1}(s)(x+y))\\
                &=s^{-p-1}(sx^{[p]}-x^{p-1}(s)x+sy^{[p]}-y^{p-1}(s)y)+\sum\lambda_i(s\inv x,s\inv y)\\
  &=(s\inv x)^{[p]}+(s\inv y)^{[p]}+\sum\lambda_i(s\inv x,s\inv y)
\end{align*}
as desired. We skip the proof that $[a^{[p]},b]=\ad(a)^p(b)$ carries over to the localization as well.

\bibliographystyle{plain}
\bibliography{eigene,andere,arxiv,mathscinet}
\end{document}